\documentclass[reqno,12pt,a4paper]{amsart}

\usepackage{amssymb}
\usepackage{bbm}
\usepackage{}

\usepackage{tikz}
\usetikzlibrary{matrix,arrows,calc,snakes,patterns,decorations.markings}

\usepackage{booktabs}
\usepackage[mathscr]{eucal}
\usepackage{graphics,epic}
\usepackage{amsfonts}
\usepackage{amscd}
\usepackage{latexsym}
\usepackage{amsmath,amssymb, amsthm, stmaryrd, bm, bbm}
\usepackage[all,2cell]{xy}
\usepackage{mathrsfs}
\usepackage{color}

\newtheorem{theorem}[subsection]{Theorem}
\newtheorem*{theorem*}{Theorem}

\newtheorem{proposition}[subsection]{Proposition}

\newtheorem*{conjecture*}{Conjecture}

\newtheorem*{question*}{Question}
\theoremstyle{remark}
\newtheorem{remark}[subsection]{Remark}
\newtheorem{example}[subsection]{Example}
\theoremstyle{definition}
\newtheorem{definition}[subsection]{Definition}
\newtheorem*{notation*}{Notation}

\newcommand{\opname}[1]{\operatorname{\mathsf{#1}}}

\renewcommand{\mod}{\opname{mod}\nolimits}

\newcommand{\Hom}{\opname{Hom}}
\newcommand{\End}{\opname{End}}

\numberwithin{equation}{section}

\begin{document}

\title[]{Derived dimension via $\tau$-tilting theory}

\author{Yingying Zhang}

\thanks{MSC2010: 16E10, 16G10, 18E30}
\thanks{Key words: support $\tau$-tilting module, endomorphism algebra, derived dimension}
\address{Department of Mathematics, Huzhou University, Huzhou 313000, Zhejiang Province, P.R.China}

\email{yyzhang@zjhu.edu.cn}
\email{}

\begin{abstract}
Comparing the bounded derived categories of an algebra and of the endomorphism algebra of a given support $\tau$-tilting module, we find a relation between the derived dimensions of an algebra and of the endomorphism algebra of a given $\tau$-tilting module.
\end{abstract}

\maketitle

\section{Introduction}\label{s:introduction}
\medskip
The dimension of a triangulated category was introduced in [11], which is a measure of the complexity of the category. It also gives a new invariant for algebras and algebraic varieties under derived equivalences. The dimension of the bounded derived category of a finite-dimensional algebra is called the derived dimension of this algebra [6]. Note that there is a close relation between the derived dimension and some other important notions, such as Loewy length, global dimension and representation dimension. Especially, the derived dimension of a finite-dimensional algebra is finite [11]. Usually it is very difficult to give the precise value of the derived dimension of an algebra. Many algebraists investigated the upper bound of the derived dimension of an algebra, see [4, 7, 10, 11, 14].

Tilting theory aims at comparing the representation theory of an algebra with that of the endomorphism algebra of a tilting module over that algebra. Also tilting theory is a rich source of derived equivalences. $\tau$-tilting theory was introduced by Adachi-Iyama-Reiten, which completes tilting theory from the viewpoint of mutation [1]. Support $\tau$-tilting modules are very closely related with some other important notions, such as torsion classes, silting modules, silting complexes and cluster-tilting objects, see [1, 2] for details.

From [1] we have known that tilting modules are exactly faithful support $\tau$-tilting modules. Hence extending results from tilting theory to $\tau$-tilting theory increases the amount of algebras that can be compared to a given one. Treffinger in [13] compared the module categories of an algebra and of the endomorphism algebra of a given support $\tau$-tilting module, and then gave a generalization of tilting theorem in the framework of $\tau$-tilting theory. Suarez investigated the relation between the global dimensions of an algebra and of the endomorphism algebra of a $\tau$-tilting module [12].

In this paper, we devote to the comparison of the bounded derived category of a given algebra with the bounded derived category of the endomorphism algebra of a given support $\tau$-tilting module. Our first result is the following.
\begin{theorem}{\rm(Theorem 3.1)}
Let $A$ be a finite-dimensional algebra and $T$ a support $\tau$-tilting $A$-module, $B=\End_{A}(T)$ be its endomorphism algebra and $C=A/\mathrm{ann}T$ the factor algebra of $A$ modulo the annihilator of $T$.
\begin{itemize}
\item[(1)] The derived functors
$$\xymatrix{
\textbf{R}\Hom_{A}(T,-):D^{b}(\opname{mod}A)\ar@<1.0mm>[r] & D^{b}(\opname{mod}B):\textbf{L}(T\otimes_{B}-)\ar@<1.0mm>[l] &
}$$
are inverse triangle-equivalences if and only if $T$ is a tilting module.
\item[(2)] The derived functors
$$\xymatrix{
\textbf{R}\Hom_{A}(T,-):D^{b}(\opname{mod}C)\ar@<1.0mm>[r] & D^{b}(\opname{mod}B):\textbf{L}(T\otimes_{B}-)\ar@<1.0mm>[l] &
}$$
are inverse triangle-equivalences.
\end{itemize}
\end{theorem}
Applying the above result, we find a relation between the derived dimensions of an algebra and of the endomorphism algebra of a given $\tau$-tilting module.
\begin{theorem}{\rm(Theorem 3.2)}
Let $A$ be a finite-dimensional algebra and $T$ a $\tau$-tilting $A$-module, $B=\End_{A}(T)$ be its endomorphism algebra. Then we have $\mathrm{der.dim}(A)\leqslant r(1+\mathrm{der.dim}(B))-1$, where the annihilator of $T$ is a nilpotent ideal with $(\mathrm{ann}T)^r=0$.
\end{theorem}
\begin{notation*}
Let $K$ be an algebraically closed field and $A$ a finite-dimensional basic $K$-algebra. We denote by $\opname{mod}A$ the category of finitely generated right $A$-modules, and by $D^{b}(\opname{mod}A)$ the bounded derived category. The Auslander-Reiten translation of $A$ is denoted by $\tau$. The annihilator of $A$-module $T$ is denoted by $\mathrm{ann}T$.
\end{notation*}

\section{Preliminaries}
\medskip
In this section, we collect some basic materials that will be used later.
\vspace{0.2cm}

{\bf 2.1. $\tau$-titing theory}

\vspace{0.2cm}

First we recall the definition of support $\tau$-tilting modules from [1].
\begin{definition}
Let $X \in \mod A$.
\begin{itemize}
\item[(1)] We call $X$ in $\mod A$ $\tau$-$rigid$ if  ${\opname{Hom}}_{A}(X, \tau X)$=0.
\item[(2)] We call $X$ in $\mod A$ $\tau$-$tilting$ if $X$ is $\tau$-rigid and $ |X| = |A|$, where $ |X|$ denotes the number of nonisomorphic indecomposable direct summands of $X$.
\item[(3)] We call $X$ in $\mod A$ $support$ $\tau$-$tilting$ if there exists an idempotent e of $A$ such that $X$ is a $\tau$-tilting ($A/\langle e\rangle$)-module.
\end{itemize}
\end{definition}
We recall some properties of support $\tau$-tilting modules associated with $\tau$-tilting modules and tilting modules.
\begin{proposition}\rm ([1, Proposition 2.2])
\begin{itemize}
\item[(1)] $\tau$-tilting modules are precisely sincere support $\tau$-tilting modules.
\item[(2)] Tilting modules are precisely faithful support $\tau$-tilting modules.
\item[(3)] Any $\tau$-tilting $A$-module $T$ is a tilting $C$-module, where $C=A/\mathrm{ann}T$.
\end{itemize}
\end{proposition}
\vspace{0.2cm}

{\bf 2.2. Derived dimension}

\vspace{0.2cm}
We recall some notions from [9, 10, 11].
Let $\mathcal {T}$ be a triangulated category and $\mathcal {I} \subseteq {\rm Ob}\mathcal {T}$.
Let $\langle \mathcal {I} \rangle$ be the smallest full subcategory of $\mathcal {T}$ containing $\mathcal {I}$ and closed under finite direct sums, direct summands and shifts. Given two subclasses $\mathcal {I}_{1}, \mathcal {I}_{2}\subseteq {\rm Ob}\mathcal {T}$, we denote $\mathcal {I}_{1}*\mathcal {I}_{2}$
by the full subcategory of all extensions between them, that is,
$$\mathcal {I}_{1}*\mathcal {I}_{2}=\{ X\mid  X_{1} \longrightarrow X \longrightarrow X_{2}\longrightarrow X_{1}[1]\;
{\rm with}\; X_{1}\in \mathcal {I}_{1}\; {\rm and}\; X_{2}\in \mathcal {I}_{2}\}.$$
Write $\mathcal {I}_{1}\diamond\mathcal {I}_{2}:=\langle\mathcal {I}_{1}*\mathcal {I}_{2} \rangle.$
Then $(\mathcal {I}_{1}\diamond\mathcal {I}_{2})\diamond\mathcal {I}_{3}=\mathcal {I}_{1}\diamond(\mathcal {I}_{2}\diamond\mathcal {I}_{3})$
for any subclasses $\mathcal {I}_{1}, \mathcal {I}_{2}$ and $\mathcal {I}_{3}$ of $\mathcal {T}$ by the octahedral axiom.
Write
\begin{align*}
\langle \mathcal {I} \rangle_{0}:=0,\;
\langle \mathcal {I} \rangle_{1}:=\langle \mathcal {I} \rangle\; {\rm and}\;
\langle \mathcal {I} \rangle_{n+1}:=\langle \mathcal {I} \rangle_{n}\diamond \langle \mathcal {I} \rangle_{1}\;{\rm for\; any \;}n\geqslant 1.
\end{align*}

\begin{definition}{\rm ([11, Definition 3.2])}
The {\bf dimension} $\dim \mathcal {T}$ of a triangulated category $\mathcal {T}$
is the minimal $d$ such that there exists an object $M\in \mathcal {T}$ with
$\mathcal {T}=\langle M \rangle_{d+1}$. If no such $M$ exists for any $d$, then we set $\dim \mathcal {T}=\infty.$
\end{definition}
Now we give the definition of derived dimension of an algebra.
\begin{definition}{\rm ([6])}
The derived dimension of $A$ is the dimension of the triangulated category $D^{b}(\opname{mod}A)$, denoted by $\mathrm{der.dim}(A)$.
\end{definition}
A finite-dimensional algebra $A$ is said to be $derived$ $finite$ if up to shift and isomorphism there is only a finite number of indecomposable objects in $D^{b}(\opname{mod}A)$ [5, Definition 1]. Clearly, $\mathrm{der.dim}(A)=0$ if and only if $A$ is derived finite. By [6], a finite-dimensional algebra over an algebraically closed field is derived finite if and only if it is an iterated tilted algebra of Dynkin type.

The following result establishes a relation between derived dimensions of an algebra and of its factor algebra.
\begin{proposition}{\rm ([10, Lemma 7.35])}
Let $A$ be an algebra and $I$ a nilpotent ideal of $A$ with $I^{r}=0$. We have $\mathrm{der.dim}(A)\leqslant r(1+\mathrm{der.dim}(A/I))-1$.
\end{proposition}

\vspace{0.2cm}
\section{Main results}
\medskip
Given an algebra $A$ and a support $\tau$-tilting $A$-module $T$, we consider the algebras $B=\End_{A}(T)$ and $C=A/\mathrm{ann}T$. Treffinger in [13] compared the module categories of $A$ and $B$, and gave a generalization of the Brenner-Butler's tilting theorem in the framework of $\tau$-tilting theory. We replace the module categories by their derived categories, and adapt the conditions accordingly, we obtain the following statement.
\begin{theorem}
Let $T$ be a support $\tau$-tilting $A$-module, $B=\End_{A}(T)$ be its endomorphism algebra and $C=A/\mathrm{ann}T$.
\begin{itemize}
\item[(1)] The derived functors
$$\xymatrix{
\textbf{R}\Hom_{A}(T,-):D^{b}(\opname{mod}A)\ar@<1.0mm>[r] & D^{b}(\opname{mod}B):\textbf{L}(T\otimes_{B}-)\ar@<1.0mm>[l] &
}$$
are inverse triangle-equivalences if and only if $T$ is a tilting module.
\item[(2)] The derived functors
$$\xymatrix{
\textbf{R}\Hom_{A}(T,-):D^{b}(\opname{mod}C)\ar@<1.0mm>[r] & D^{b}(\opname{mod}B):\textbf{L}(T\otimes_{B}-)\ar@<1.0mm>[l] &
}$$
are inverse triangle-equivalences.
\end{itemize}
\end{theorem}
\begin{proof}
(1) By [8], the derived functors $\textbf{R}\Hom_{A}(T,-)$ and $\textbf{L}(T\otimes_{B}-)$ are inverse triangle-equivalences if and only if $T$ is a tilting module of finite projective dimension. Thus $T$ is a faithful support $\tau$-tilting module. It follows that $T$ is a tilting module by Proposition 2.2(2).

(2) Because of Proposition 2.2(3) and [1, Proposition 1.1 and Theorem 2.7], $T$ is a $C$-tilting module. Moreover since $\mod C$ is a full subcategory of $\mod A$, we have that $\End_{C}(T)\cong \End_{A}(T)=B$. Therefore it follows from [8] that $C$ is derived equivalent to $B$.
\end{proof}

Assume that $T$ is a $\tau$-tilting module, Suarez investigated the relation between the global dimensions of $A$ and of $B$ in [12]. Here we replace the global dimensions by their derived dimensions and get the following result.
\begin{theorem}
Let $A$ be an algebra and $T$ a $\tau$-tilting $A$-module, $B=\End_{A}(T)$ be its endomorphism algebra. Then we have $\mathrm{der.dim}(A)\leqslant r(1+\mathrm{der.dim}(B))-1$, where $\mathrm{ann}T$ is a nilpotent ideal with $(\mathrm{ann}T)^r=0$.
\end{theorem}
\begin{proof}
Since $T$ is a $\tau$-tilting module, by Proposition 2.2(1) $T$ is sincere. It follws from [3] that $\mathrm{ann}T$ is a nilpotent ideal. By Theorem 3.1(2), $\mathrm{der.dim}(B)=\mathrm{der.dim}(A/\mathrm{ann}T)$. Applying Proposition 2.5, we have $$\mathrm{der.dim}(A)\leqslant r(1+\mathrm{der.dim}(A/\mathrm{ann}T))-1$$.
\end{proof}
\begin{remark}
\begin{itemize}
\item[(1)] It is obvious that the equality holds if $T$ is a tilting module.
\item[(2)] The derived dimension of $A$ can be strictly less than $r(1+\mathrm{der.dim}(B))-1$. For example, let $A$ be an iterated tilted algebra of Dynkin type and $T$ a $\tau$-tilting module which is not tilting.
\end{itemize}
\end{remark}
In fact, the bound of Theorem 3.2 is sharp, even $T$ is $\tau$-tilting which is not tilting. We illustrate this on a simple but non trivial example.
\begin{example}
Let $A$ be an radical square zero algebra with the following quiver Q:
$$\xymatrix{
1 \ar[r]^{\alpha }&
2 \ar@(ur,dr)^{\beta }&
}$$
Observe that $\mathrm{der.dim}(A)=1$. The support $\tau$-tilting quiver of $A$ is the following:\\
$$\begin{xy}
(0,0)*+{\begin{smallmatrix} 1\\ 2\end{smallmatrix} \begin{smallmatrix} 2\\2\end{smallmatrix}}="1",
(0,-20)*+{\begin{smallmatrix} 2\\2\end{smallmatrix}}="2",
(20,0)*+{\begin{smallmatrix}\color{red}{1}\\ \color{red}{2}\end{smallmatrix} \begin{smallmatrix} \color{red}{1}\end{smallmatrix}}="3",
(40,0)*+{\begin{smallmatrix} 1\end{smallmatrix}}="4",
(40,-20)*+{\begin{smallmatrix} 0\end{smallmatrix}}="5",
\ar"1";"2",\ar"1";"3",\ar"3";"4",\ar"2";"5",
\ar"4";"5",
\end{xy}$$
Consider $T={\begin{smallmatrix}1\\ 2\end{smallmatrix} \begin{smallmatrix} 1\end{smallmatrix}}$ a unique $\tau$-tilting $A$-module which is not tilting. In this case, $\mathrm{ann}T=\langle \beta\rangle $ and $(\mathrm{ann}T)^{2}=0$, i.e., $r=2$.

The endomorphism algebra $B=\End_{A}(T)$ is the hereditary algebra with quiver $1\longleftarrow 2$. Then $\mathrm{der.dim}(B)=0$.

Therefore $1=\mathrm{der.dim}(A)\leqslant r(1+\mathrm{der.dim}(B))-1=1$.
\end{example}

\bigskip
{\bf Acknowledgement.} The author would like to thank Xiaowu Chen, Junling Zheng and Yu Zhou for their helpful discussions. The author also thanks the referee for very useful suggestions concerning the presentation of this paper.

\end{document}